\documentclass[1p]{article} 
\usepackage{amsmath,amsfonts,amsthm} 
\usepackage{subfigure}
\usepackage{latexsym}
\usepackage{amssymb}
\usepackage[dvips]{graphics}
\usepackage[dvips]{graphicx}
\newtheorem{theorem}{Theorem}
\newtheorem{proposition}[theorem]{Proposition}
\newtheorem{lemma}[theorem]{Lemma}
\newtheorem{corollary}[theorem]{Corollary}
\newtheorem{definition}[theorem]{Definition}
\newtheorem{remark}[theorem]{Remark}
\newtheorem{example}[theorem]{Example}

\newcommand{\R}{{\mathbb R}}

\newcommand{\Q}{{\mathbf Q}}

\renewcommand{\int}{\rm Int}
\newcommand{\E}{\mathbb E}

\newcommand{\PP}{\Bbb {P}}
\newcommand{\lk}{{\rm {Lk}}}

\newcommand{\D}{{\mathfrak D}}
\newcommand{\T}{{\mathcal T}}

\newcommand{\p}{{\mathfrak p}}

\newenvironment{dedication}
    {\vspace{6ex}\begin{quotation}\begin{center}\begin{em}}
    {\par\end{em}\end{center}\end{quotation}}

\begin{document}

\title{Large Random Simplicial Complexes, III\\
The Critical Dimension}          
\author{A. Costa and M. Farber}        
\date{December 19, 2015}          
\maketitle

\begin{abstract}
In this paper we study the notion of critical dimension of random simplicial complexes in the general multi-parameter 
model described in \cite{CF14}, \cite{CF15}, \cite{CF15a}. This model includes as special cases the Linial-Meshulam-Wallach model \cite{LM}, \cite{MW} as well as the clique complexes of random graphs. 
We characterise the concept of critical dimension in terms of various geometric and topological properties of random simplicial complexes such as their Betti numbers, the fundamental group, the size of minimal cycles and the degrees of simplexes. We mention in the text a few interesting open questions.
\end{abstract}

\begin{dedication}
To the memory of Tim Cochran
\end{dedication}

\section{Introduction}

This paper continues our recent publications \cite{CF14}, \cite{CF15}, \cite{CF15a}, in which we introduced a new class of random simplicial complexes depending on many
probability parameters $p_0, p_1, \dots, p_r$; here $p_i\in [0,1]$ determines the probability that an $i$-dimensional simplex 
$\sigma$ is included in our random simplicial complex $Y$ given that the boundary $\partial \sigma$ is included into 
$Y$. This model of random simplicial complexes includes as special cases many previously studied models such as the 
Erd\"os-R\'one random graphs \cite{ER}, the Linial - Meshulam - Wallach random simplicial complexes \cite{LM}, \cite{MW} as well as the clique complexes of random graphs \cite{Kahle1}, \cite{CFH}. 

The main motivation to study large random simplicial complexes comes from the need of modelling large complex systems in various applications 
when the classical notion of configuration space becomes inadequate. These applications include the new theory of large networks \cite{Newman}.
Recent surveys covering various aspects of topology of large random spaces can be found in \cite{CFK} and \cite{Ksurvey}. 

This paper focuses on the notion of a {\it critical dimension} of random simplicial complexes. 
We show that this notion influences the topology of random complexes in a number of different ways. Firstly, the Betti number in the critical dimension significantly dominates the Betti numbers in all other dimensions (see Theorem \ref{thmdom}, statement (2)) and the reduced Betti numbers below the critical dimension vanish (see Theorems \ref{thmgarland} and \ref{k1}. 
Secondly, most simplexes $\sigma \subset Y$ of a random simplicial complex $Y$ in dimensions at or above the critical dimension have degree zero while simplexes below the critical dimension have their degrees unbounded (see Corollary \ref{degree}). 
Thirdly,  above the critical dimension any minimal cycle has a bounded size (Theorem \ref{smallcycles}); we use the Dehn-Sommerville relations to establish an upper bound on the size of spheres 
above the critical dimension which can be embedded into random simplicial complexes (Theorem \ref{smallspheres}). Besides, we show that 
 the fundamental group of a random simplicial complex has property (T) if the critical dimension is $\ge 2$ and, 
moreover, a random simplicial complex is simply connected if the critical dimension is $\ge 3$, see Theorems \ref{k2} and 
\ref{kgreater3}.

There is a small overlap between the results presented in this paper and the preprint of C. Fowler \cite{Fowler} although these two papers largely complement each other.

 This research was supported by a grant from the EPSRC research council.




\section{The model}

Here we briefly recall the model of random simplicial complexes we studied in \cite{CF14}, \cite{CF15}, \cite{CF15a}. 

Let $\Delta_n$ denote the simplex with the vertex set $\{1, 2, \dots, n\}$.
We view $\Delta_n$ as an abstract simplicial complex of dimension $n-1$.
For a simplicial subcomplex $Y\subset \Delta_n$, we denote by $f_i(Y)$ the number of {\it $i$-faces} of $Y$ (i.e. $i$-dimensional simplexes of $\Delta_n$ contained in $Y$).
An \textit{external face} of a subcomplex $Y\subset \Delta_n$ is a simplex $\sigma \subset \Delta_n$ such that $\sigma \not\subset Y$ but the boundary of $\sigma$ is contained in $Y$,
$\partial \sigma \subset Y$. The symbol $e_i(Y)$ will denote the number of
$i$-dimensional external faces of $Y$.



%

Fix an integer $r\ge 0$ and a sequence ${\mathfrak p}=(p_0, p_1, \dots, p_r)$ of real numbers satisfying $0\le  p_i\le 1.$ Denote
$q_i=1-p_i.$
We consider the probability space ${\Omega_n^r}$ consisting of all subcomplexes
$Y\subset \Delta_n,$ with $ \dim Y\le r.$
The probability function
\begin{eqnarray}
\PP_{r,\p}: {\Omega_n^r}\to \R
\end{eqnarray} is given by the formula
\begin{eqnarray}\label{def1}
\PP_{r, \p}(Y) \, 
&=& \, \prod_{i=0}^r p_i^{f_i(Y)}\cdot \prod_{i=0}^r q_i^{e_i(Y)}
\end{eqnarray}
In (\ref{def1}) we use the convention $0^0=1$; in other words, if $p_i=0$ and $f_i(Y)=0$ then the corresponding factor in (\ref{def1}) equals 1; similarly if some $q_i=0$ and $e_i(Y)=0$.
One may show that $\PP_{r, \p}$ is indeed a probability function, i.e.
$$
\sum_{Y\subset \Delta_n^{(r)}}\PP_{r, \p}(Y) = 1,
$$
see \cite{CF14}.
Note that the probability $\PP_{r,\p}(Y)$ really depends on $r$. For example if $\dim Y=r-1$ then 
$$\PP_{r, \p}(Y) = \PP_{r-1, \p'}(Y)\cdot q_r^{e_r(Y)}$$
where $\p'=(p_0, \dots, p_{r-1})$.


%
%
%
%
%

%
%
%
%


\section{The notion of a critical dimension}

In this section we introduce the notion of a critical dimension of a multi-parameter random simplicial complex. 
In the following sections we shall explain the role this notion plays in topology of random simplicial complexes.

First we define the following domains $$\D_{-1}, \D_0, \D_1,\dots, \D_r\subset \R_+^{r+1}.$$
Consider the linear functions $$\psi_k: \R^{r+1}\to \R$$ where for $\alpha = (\alpha_0, \dots, \alpha_r)\in \R^{r+1}$ one has
\begin{align*}
\psi_k(\alpha)=\sum_{i=0}^r \binom{k}{i}\alpha_i,\quad\quad k=0,\ldots, r.
\end{align*} 
We use the conventions that $\binom{k}{i}=0$ for $i>k$ and $\binom{0}{0}=1$.

We assume below that $\alpha\in \R_+^{r+1}$, i.e. $\alpha=(\alpha_0, \dots, \alpha_r)$ with $\alpha_i\ge 0$ for all $i$.
Since $\binom k i < \binom {k+1} i$ for $i>0$ we see that
\begin{align}\label{ineq1}
\psi_0(\alpha)\leq\psi_1(\alpha)\leq\psi_2(\alpha)\leq\ldots\leq\psi_r(\alpha).
\end{align}
Moreover, if for some  $j\geq 0$ one has $\psi_j(\alpha)<\psi_{j+1}(\alpha)$ then
\begin{equation*}\label{ineq2}\psi_j(\alpha)<\psi_{j+1}(\alpha)<\ldots<\psi_r(\alpha).
\end{equation*}

Next we define the following convex domains in $\R^{r+1}_+$:
\begin{eqnarray}\label{defdk}
{\mathfrak D}_k = \{\alpha\in \R^{r+1}_+\, ;\,  \psi_k(\alpha)<1<\psi_{k+1}(\alpha)\},
\end{eqnarray}
where $k=0, 1, \dots, r-1$. 
One may also introduce the domains
$$\D_{-1}=  \{\alpha\in \R^{r+1}_+\, ;\,  1<\psi_0(\alpha)\}, \quad \D_{r}=  \{\alpha\in \R^{r+1}_+\, ;\,  \psi_r(\alpha)<1\}.$$
The domains $$\D_{-1}, \D_0, \D_1,\dots, \D_r$$ are disjoint and their union is $$\bigcup_{j=-1}^r \D_j \, = \, \R^{r+1}_+- \bigcup_{i=0}^r H_i,$$
where $H_i$ denotes the hyperplane 
$$H_i=\{\alpha\in \R^{r+1}; \psi_i(\alpha)=1\}.$$ 
We shall show that hyperplanes $H_0, \dots, H_r$
correspond to {\it phase transitions in homology}; in other words these hyperplanes separate areas 
where different geometric and topological properties are satisfied for large $n$, with high probability. 
The intersections of the domains $\D_k$ with the plane $\alpha_0=\alpha_3=\dots=0$ is shown on the Figure \ref{domains}.

 \begin{figure}[h]
\centering
\includegraphics[width=0.65\textwidth]{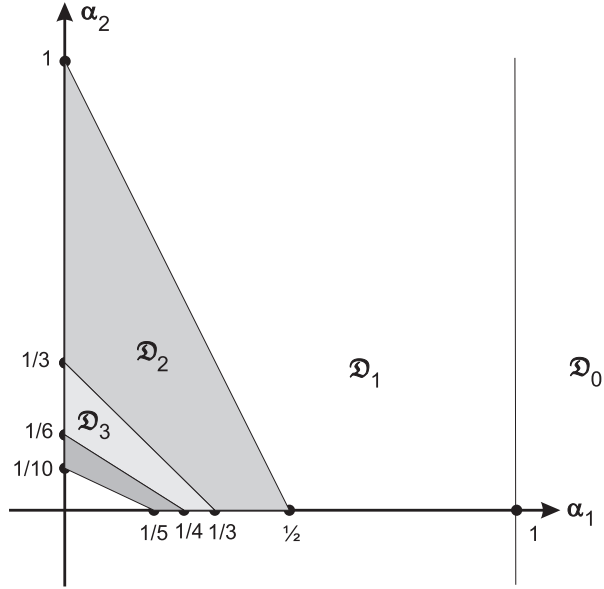}
\caption{Intersections of the domains $\D_k$ with the plane $\alpha_0=\alpha_3=\dots=0$. }\label{domains}
\end{figure}

Like in our previous publications \cite{CF14}, \cite{CF15}, \cite{CF15a}, we shall consider here the multi-parameter random simplicial complexes with the multi-parameters $p_i=n^{-\alpha_i}$ where 
in general the vector  of exponents
$$\alpha \, =\, (\alpha_0, \dots, \alpha_r)\in \R_+^{r+1}$$ is a function of $n$, i.e. $\alpha=\alpha(n)$. In this paper 
we shall additionally assume that the vector of exponents
$\alpha=\alpha(n)$ is either constant or more generally has a limit $$\alpha_\ast=\lim \alpha(n).$$

\begin{definition}\label{critical} Under the above assumptions we shall say that  the critical dimension of random simplicial complex $Y\in\Omega_n^r$ 
equals $k$ if 
\begin{eqnarray}\label{crit}
\alpha_\ast\in \D_k,\quad \mbox{where}\quad k=-1, 0, \dots, r.\end{eqnarray}
\end{definition}

\begin{example} {\rm Assume that the vector of exponents $\alpha$ is constant and $\alpha\in \D_{-1}$. Then $p_0=n^{-\alpha_0}$ with $\alpha_0>1$. 
We know that in this case the random complex $Y$ is empty, $Y=\emptyset$, a.a.s. 
See \cite{CF14}, Example 2.4. 
}
\end{example}

\begin{example}{\rm 
Let us now suppose that the vector of exponents $\alpha$ lies in $\D_0$ and is constant (i.e. is independent of $n$). 
Then $$\alpha_0<1<\alpha_0+\alpha_1.$$ 
By Lemma 2.7 from \cite{CF15} the number of vertices of a random complex $Y\in \Omega_n^r$  is close to  
$n^{1-\alpha_0}\to \infty$ and by Example 6.3 from \cite{CF15} the complex $Y$ is disconnected. On the other hand, 
the expected number of 
one-dimensional cycles of any lengths $k\ge 3$ in $Y$ is 
$$\le \sum_{k=3}^\infty (np_0)^kp_1^k\, = \, \sum_{k=3}^\infty n^{k(1-\alpha_0-\alpha_1)} \le 2\cdot n^{3(1-\alpha_0-\alpha_1)}\to 0.$$
Using the first moment method, we obtain that {\it  for $\alpha\in \D_0$, with probability tending to 1, 
the random complex $Y$ is a forest with many connected components. }
}
\end{example}

\begin{example}{\rm 
Suppose that the vector of exponents $\alpha$  is constant and lies in $\D_1$. 
Then $$\alpha_0+\alpha_1<1< \alpha_0+2\alpha_1+\alpha_2.$$ 
By Example 7.3 from \cite{CF15} a random complex $Y\in \Omega_{n}^r$ is connected, a.a.s. }
\end{example}

\section{The homological domination principle}

In this section we state a theorem stating that for a random complex $Y\in \Omega_n^r$,
the Betti number in the critical dimension $k$ is significantly larger than any other Betti number
 $b_j(Y)$ where $j\neq k$, a.a.s; see Theorem \ref{thmdom} below. 
 
 We show later in this paper that the reduced Betti numbers $\tilde b_j(Y)$ vanish for $j<k$, see Theorem \ref{thmgarland}.


We start by defining the following affine functions
\begin{eqnarray}\label{tau}
\tau_k(\alpha)\, =\, k+1-\sum_{i=0}^\ell \psi_i(\alpha) \, =\, \sum_{i=0}^k\left[1-\psi_i(\alpha)\right],
\end{eqnarray}
where $\alpha=(\alpha_0, \dots, \alpha_r)\in \R^{r+1},$  and $k=0,\ldots, r.$

\begin{remark} \label{rmk1} {\rm Note that for $\alpha\in \D_{-1}$ one has $\psi_0(\alpha)>1$ and hence $\psi_i(\alpha)>1$ for any $i$; therefore one has 
$\tau_i(\alpha)<0$ for any $i=0, 1, \dots, r$. }
\end{remark}

\begin{remark} \label{rmk2}{\rm If $\alpha\in \D_k$ where $k\ge 0$ then $\psi_i(\alpha)<1$ for all $i\le k$ and hence $\tau_i(\alpha)>0$ for all $i=0, \dots, k$.
Moreover, if $\alpha\in \D_k$ where $k\ge 0$ then 
\begin{eqnarray*}
0<\tau_0(\alpha) <\tau_1(\alpha)<\dots< \tau_k(\alpha)\quad \mbox{and} \quad
\tau_k(\alpha)>\tau_{k+1}(\alpha) >\dots > \tau_r(\alpha).
\end{eqnarray*}
}\end{remark}

%

\begin{theorem}\label{thmdom} Consider a multi-parameter random simplicial complex $Y\in \Omega_n^r$ with respect to the probability measure
$$\PP_{r, \p}:\Omega_n^r\to \R,\quad \mbox{where } \quad
\p = n^{-\alpha},$$ i.e. $$\p=(p_0, p_1, \dots, p_r), \qquad p_i=n^{-\alpha_i}.$$
Here $\alpha=\alpha(n)$ is a function of $n$. We assume that the limit $\lim \alpha(n) =\alpha_\ast\in \R^{r+1}_+$ exists and 
\begin{eqnarray}\label{exists}
\alpha_\ast\in \D_k, 
\end{eqnarray}
for some $k=0, 1, \dots, r,$ i.e. $k$ is the critical dimension, see Definition \ref{critical}. 
Then:
%
%
\begin{enumerate}
\item For a sequence $t\to 0$, a random complex $Y\in \Omega_n^r$ satisfies
\begin{eqnarray}\label{firstly}
(1-t)\cdot\frac {n^{\tau_k(\alpha)}}{(k+1)!}\, \leq\,  b_k(Y)\, \leq\, (1+t)\cdot \frac {n^{\tau_k(\alpha)}}{(k+1)!},\quad a.a.s.
\end{eqnarray}

\item For any $j\not= k$ with probability tending to one,
\begin{eqnarray}\label{dombeta1}
b_j(Y) \le (1+o(1)) \cdot (r+1)!\cdot n^{-e(\alpha)} \cdot b_k(Y), 
\end{eqnarray}
where
\begin{eqnarray}\label{ealpha}
e(\alpha) &=& \min_{0\leq s\leq r} \{|1-\psi_s(\alpha)|\}\, \nonumber\\
&=&\, \min\{1-\psi_k(\alpha), \psi_{k+1}(\alpha)-1\}.
\end{eqnarray}
\end{enumerate}
\end{theorem}

Note that $$\lim e(\alpha) =e(\alpha_\ast) >0.$$ Observe that (\ref{dombeta1}) implies that for $j\not=k$ one has 
$$b_j(Y) \le 2 (r+1)!\cdot n^{-e(\alpha_\ast)/2} \cdot b_k(Y),$$
a.a.s.
Formula (\ref{dombeta1}) shows that for $\alpha_\ast \in {\mathfrak D}_k$ the $k$-th Betti number dominates all other Betti numbers.

The proof of Theorems \ref{thmdom} is given in \S \ref{prfdom}.

\section{Face numbers of random complexes}

Consider the number $f_d(Y)$ of $d$-dimensional simplexes of $Y$ where  $Y\in \Omega_n^r$ is a random simplicial complex, $0\le d\le r$. 
The function 
$f_d: \Omega_n^r\to \R$ is a random variable and in this section we show that for large $n$ the value of this function is very close to its expectation 
on a large set of random simplicial complexes. 



\begin{theorem}\label{thmdim}
Let $Y\in \Omega_n^r$ be a multi-parameter random simplicial complex with multi-parameter 
$\p=(p_0, p_1, \dots, p_r),\quad p_i=n^{-\alpha_i},$ 
where in general $\alpha_i=\alpha_i(n)\geq 0$ is a function of $n$. 
Assume that the limit 
\begin{eqnarray}\label{limit}
\alpha_*=\lim_{n\to\infty}\alpha(n)
\end{eqnarray}
exists. Then the following statements are true:

{\rm (a)} If for an integer $0\leq d\leq r$ one has\footnote{Recall that the function $\tau_d(\alpha)$ is defined by (\ref{tau}).} 
$\tau_d(\alpha_*)<0$ then $$f_d\equiv0,\quad a.a.s.$$ 
 In other words, in this case $\dim Y<d$, a.a.s.

{\rm (b)} If for an integer $0\leq d\leq r$ one has  $\tau_d(\alpha_*)=0$ then for any real $\mu>0$ one has 
$$f_d(Y) < n^\mu, \quad \mbox{a.a.s.}$$

{\rm (c)} If for an integer $0\leq d\leq r$ one has 
\begin{eqnarray}\label{pos} \tau_d(\alpha_*)>0
\end{eqnarray} 
then for any sequence\footnote{In particular, one may take $t=\log n \cdot n^{-\delta_d(\alpha_\ast)/4}$.} 
$t=t(n)$ such that $tn^{\delta_d(\alpha_\ast)/4}\to \infty$, where\footnote{Recall that $\tau_d(\alpha_\ast)>0$ implies $\tau_0(\alpha_\ast)>0$, see Remark 
\ref{rmk2}.}
\begin{eqnarray}\label{delta}
\delta_d(\alpha_*)=\min \left\{\tau_0(\alpha_*),\tau_d(\alpha_*) \right\}>0,
\end{eqnarray}
the $\PP_{r,\p}$-probability that a random complex $Y\in \Omega_n^r$ satisfies
\begin{eqnarray}\label{ineq17}
\left(1-t\right)\cdot \frac{n^{\tau_d(\alpha)}}{(d+1)!}\leq f_d(Y) \leq (1+t)\frac{n^{\tau_d(\alpha)}}{(d+1)!}
\end{eqnarray}
tends to $1$ as $n\to \infty$.
\end{theorem}

\begin{remark}{ \rm  It follows from Theorem \ref{thmdim} that 
$$\frac{\log f_d(Y)}{\log n} \sim \tau_d(\alpha_\ast), \quad \mbox{a.a.s.}$$

}
\end{remark}

\begin{proof}[Proof of Theorem \ref{thmdim}]
For an integer $0\leq d\leq r$, let $\sigma\subset \Delta_n^{(r)}$ be a fixed $d$-dimensional simplex.
By Lemma 2.9 in \cite{CF15} one has
\begin{eqnarray}\label{simplex}
\PP_{r, \p}(\sigma\subset Y)=\prod_{i=0}^{d}p_i^{f_i(\sigma)}=n^{-\sum_{i=0}^d \binom{d+1}{i+1}\alpha_i}.
\end{eqnarray}
We shall use the identity 
\begin{eqnarray}\label{sumpsi1}\sum_{i=0}^d\binom{d+1}{i+1}\alpha_i=\sum_{k=0}^d\psi_k(\alpha),\quad \mbox{where} \quad 0\leq d\leq r,
\end{eqnarray}
which is equivalent to the well-known identity 
$\binom {d+1} {i+1} = \sum_{k=i}^d \binom k i.$
The LHS of this equality counts the number of $(i+1)$-element subsets of the set $\{1, \dots, d+1\}$; for each $k=i, \dots, d$ the number 
$\binom k i$ counts the number of $(i+1)$-element subsets of the set $\{1, \dots, d+1\}$ for which $k+1$ is the maximal element. 

By the definition of the affine function $\tau_d(\alpha)$ we have $\sum_{k=0}^d\psi_k(\alpha)=d+1-\tau_d(\alpha)$ and therefore (\ref{simplex}) can be rewritten as 
$\PP_{r, \p}(\sigma\subset Y)=n^{\tau_d(\alpha)-(d+1)}$ and hence 
\begin{eqnarray}\label{efd}
\E(f_d) = \, {\binom n {d+1}} \cdot \prod_{i+0}^{d+1} p_i^{\binom {d+1}{i+1}} \, =\, \binom n {d+1} n^{\tau_d(\alpha)-(d+1)}.
\end{eqnarray}
Thus we obtain
\begin{eqnarray}\label{ex1}
\left(1-\frac{d^2}{n}\right) \cdot\frac{n^{\tau_d(\alpha)}}{(d+1)!}\leq \E(f_d)\leq \frac{n^{\tau_d(\alpha)}}{(d+1)!}
\end{eqnarray}
for $n$ large enough. 


To prove statement (a), assume that 
$\lim_{n\to\infty}\tau_d(\alpha)=\tau_d(\alpha_*)<0$ is negative. 
It follows that for $n$ large enough we have $\tau_d(\alpha)<\tau_d(\alpha_\ast)/2<0$ and hence
$$\E(f_d)\leq \frac{n^{\tau_d(\alpha_*)/2}}{(d+1)!}.$$ 
Thus we see that $\E(f_d)\to 0$ as $n\to\infty$ and by the Markov inequality we obtain that 
$f_d(Y)=0$,
asymptotically almost surely. 

To prove statement (b), we assume that $\tau_d(\alpha_\ast)=0$ and observe that for any $\mu>0$ one has 
$\tau_d(\alpha)<\mu/2$ for all $n$ large enough. Then using the Markov inequality we find
$$\PP_{r, \p}(f_d(Y) \ge n^\mu) \le \frac{\E(f_d)}{n^\mu} \, \le\,  \frac{n^{\tau_d(\alpha)-\mu}}{(d+1)!}\, \le\,  \frac{n^{-\mu/2}}{(d+1)!}\, =\, o(1).$$

Next we prove statement (c) assuming that $\tau_d(\alpha_*)>0$. From (\ref{ex1}) we obtain 
\begin{align}
\E(f_d)\geq \left(1-\frac{d^2}{n}\right)\cdot \frac{n^{\tau_d(\alpha_*)/2}}{(d+1)!}\, \to\, \infty
\end{align}
as $n\to\infty$.
We shall show below that for $n$ large enough one has
\begin{align}\label{var1}
\frac{{\rm Var}(f_d)}{\E(f_d)^2}\leq C\cdot n^{-\delta_d(\alpha_*)/2}
\end{align}
where $C$ is a constant determined by $d$ and $\delta_d(\alpha_*)$ is given by (\ref{delta}). 
Let $t=\omega n^{-\delta_d(\alpha_\ast)/4}$ where $\omega \to \infty$.
Applying the Chebychev inequality we obtain 
 $$\PP_{r,\p}\left( |f_d-\E(f_d)|\geq t\cdot\E(f_d) \right)\, \leq\,  \frac{{\rm Var}(f_d)}{t^2\cdot\E(f_d)^2}\, \leq\,  \frac{C}{t^2 n^{\delta_d(\alpha_*)/2}}=o(1).$$
Hence for any sequence $t=t(n)$ as above we shall have
$$(1-t)\cdot\E(f_d)\, \leq\,  f_d\, \leq\,  (1+t)\cdot\E(f_d),\quad \mbox{a.a.s.}$$
which together with (\ref{ex1}) imply 
$$
\left(1-t\right)\cdot\left(1-\frac{d^2}{n}\right)\cdot \frac{n^{\tau_d(\alpha)}}{(d+1)!}\leq f_d \leq (1+t)\frac{n^{\tau_d(\alpha)}}{(d+1)!}, \, \quad \mbox{a.a.s.}.
$$
Next we observe that $\delta_d(\alpha_\ast)\le 1$ and therefore $tn^{1/4}\to \infty$. This implies the inequality $(1-t)(1-\frac{d^2}{n}) \ge (1-2t)$ for $n$ large enough. This gives the left inequality in 
(\ref{ineq17}) by applying the whole argument to $t/2$ instead of $t$. 
%

We are left to prove the inequality (\ref{var1}).
Given a $d$-dimensional simplex $\sigma\subset \{1,\ldots,n\}$ we denote by $A_\sigma$ the event that $\sigma$ is a simplex of 
$Y\in \Omega_n^r$. Given two $d$-simplexes $\sigma,\, \tau\subset\{1,\ldots,n\}$ we denote by $A_\sigma\wedge A_\tau$ the event 
$\sigma\cup\tau\subset Y$.
We have (see Lemma 2.2 from \cite{CF15})
$$\PP_{r,\p}(A_\sigma\wedge A_\tau)=\prod_{i=0}^d p_i^{2\binom{d+1}{i+1}-\binom{f_0(\sigma\cap\tau)}{i+1}}$$
where $f_0(\sigma\cap\tau)$ denotes the number of vertices common to $\sigma$ and $\tau$. 

We obtain
\begin{eqnarray*}
\E(f_d^2)&=&\sum_{\sigma^{(d)}, \tau^{(d)}}\PP_{r, \p}(A_\sigma\wedge A_\tau)\\
&=& \sum_{j=0}^{d+1} \binom n {d+1} \cdot \binom {d+1} j\cdot \binom {n-d-1}{d+1-j} 
\cdot \prod_{i=0} ^{d+1} p_i^{2{\binom {d+1} {i+1}}- \binom j {i+1}}\\
\end{eqnarray*}
Using (\ref{efd}) we find
\begin{eqnarray*}
\frac{\E(f_d^2)}{\E(f_d)^2} &=& {\binom n {d+1}}^{-1} \cdot \left[ \binom {n-d-1} {d+1} + \sum_{j=1}^{d+1} \binom {d+1} j\cdot \binom {n-d-1}{d+1-j}\cdot \prod_{i=0}^{j-1} p_i^{-\binom j {i+1}}\right]\\
&=& {\binom n {d+1}}^{-1} \cdot  \left[\binom {n-d-1} {d+1} + \sum_{j=1}^{d+1} \binom {d+1} j\cdot \binom {n-d-1}{d+1-j}\cdot n^{\sum_{i=0}^{j-1}\binom j {i+1} \alpha_i}\right]\\
&=& {\binom n {d+1}}^{-1} \cdot \left[\binom {n-d-1} {d+1} +
\sum_{j=1}^{d+1} \binom {d+1} j\cdot \binom {n-d-1}{d+1-j}\cdot n^{\sum_{k=0}^{j-1}\psi_k(\alpha)}\right]\\
&\le & 1 + {\binom n {d+1}}^{-1}\cdot \sum_{j=1}^{d+1} \binom {d+1} j\cdot \binom {n-d-1}{d+1-j}\cdot n^{\sum_{k=0}^{j-1}\psi_k(\alpha)}.
\end{eqnarray*}
On the third step we have used formula (\ref{sumpsi1}). 
Next we observe that 
$${\binom n {d+1}}^{-1}\cdot \binom {d+1} j\cdot \binom {n-d-1}{d+1-j}\le C_d\cdot n^{-j}$$
where $C_d$ is a constant depending on $d$ but not on $n$. Hence we may write
\begin{eqnarray*}
\frac{\E(f_d^2)}{\E(f_d)^2} &\le & 1+ C_d\cdot \sum_{j=1}^{d+1} n^{-j+\sum_{k=0}^{j-1} \psi_k(\alpha)}\\
&=& 1+ C_d\cdot \sum_{\ell=0}^d n^{-\tau_\ell(\alpha)}.
\end{eqnarray*}
Using the convexity properties of $\tau_\ell(\alpha)$ (see Remark \ref{rmk2}) one has (for any $n$)
\begin{eqnarray*}
\min_{\ell \le d} \tau_\ell (\alpha) &=& \min\{ \tau_0(\alpha), \tau_d(\alpha)\} \\
&\ge & \frac{1}{2} \min\{ \tau_0(\alpha_\ast), \tau_d(\alpha_\ast)\} = \frac{1}{2} \delta_d(\alpha_\ast)>0. 
\end{eqnarray*}
Thus we obtain 
$$
\frac{{\rm Var}(f_d)}{\E(f_d)^2} = \frac{\E(f_d^2)}{\E(f_d)^2} -1 \le C'_d n^{-\delta_d(\alpha_\ast)/2}$$
which implies (\ref{var1}). Here $C_d'=(d+1)C_d$ is a constant depending on $d$. 
This completes the proof of Theorem \ref{thmdim}. 
\end{proof}

\section{Proof of Theorem \ref{thmdom}.}\label{prfdom}

Every simplicial complex $Y\in \Omega_n^r$ satisfies the Morse inequalities
\begin{eqnarray}\label{ineqs1}
f_j(Y)-f_{j+1}(Y)-f_{j-1}(Y)\, \leq\,  b_j(Y) \, \leq\,  f_j(Y),
\end{eqnarray}
where $0\leq j\leq r$. We shall use (\ref{ineqs1}) together with Theorem \ref{thmdim} to obtain information about the Betti numbers of $Y$.

We start with the following three observations: 

(1) If $\tau_j(\alpha_\ast)<0$ then $f_j(Y)=0$ and hence $b_j(Y)=0$, a.a.s. (see Theorem \ref{thmdim}, part (a)). 

(2) If $\tau_j(\alpha_\ast)=0$ then for any $\mu>0$ one has $$b_j\, (Y) \le\,  f_j(Y)\, \le\,  n^\mu,$$ a.a.s. (see Theorem \ref{thmdim}, part (b)). 

(3) Let us assume that $\tau_j(\alpha_\ast)>0$ and 
let $t=t(n)$ be any sequence of real numbers satisfying $$tn^{\delta_j(\alpha_\ast)/4}\to \infty$$ where $\delta_j(\alpha_\ast)= 
\min\{\tau_0(\alpha_\ast), \tau_j(\alpha_\ast)\}>0.$ Then by Theorem \ref{thmdim}, part (c) we have 
\begin{eqnarray*}
b_j(Y) \, \le \, f_j(Y)\, \le (1+t)\cdot \frac{n^{\tau_j(\alpha)}}{(j+1)!},\, 
\quad \mbox{a.a.s.}
\end{eqnarray*}

Now, suppose that $k\geq 0$ is such that the vector $\alpha_*\in\D_k$ lies in the domain $\D_k$; in other words, we assume that $k$ 
is the critical dimension of the random complex $Y$. Then $\tau_k(\alpha_\ast)>0$ and as above we have 
$$b_k(Y) \, \le \, (1+t)\cdot \frac{n^{\tau_k(\alpha)}}{(k+1)!},\quad \mbox{a.a.s.}$$
for a sequence $t=t(n)$ tending to $0$. 
Besides, since $\tau_{k-1}(\alpha_\ast)>0$ we have 
\begin{eqnarray*}f_{k-1}(Y) \, &\le& \, (1+t')\cdot \frac{n^{\tau_{k-1}(\alpha)}}{k!} \\
&\le& (1+t') \cdot \frac{n^{\tau_k(\alpha)-e(\alpha)}}{k!}
\end{eqnarray*}
a.a.s. for a sequence $t'\to 0$. Similarly, if $\tau_{k+1}(\alpha_\ast)>0$ we similarly obtain 
\begin{eqnarray*}f_{k+1}(Y) 
\, \le \, (1+t'') \cdot \frac{n^{\tau_k(\alpha)-e(\alpha)}}{(k+2)!}
\end{eqnarray*}
a.a.s. for a sequence $t''\to 0$. The last inequality still holds in the case when $\tau_{k+1}(\alpha_\ast)\le 0$ by statements (a) and (b) of Theorem \ref{thmdim} since $\tau_k(\alpha)-e(\alpha)$ converges to $\tau_k(\alpha_\ast) - e(\alpha_\ast)>0$. To explain the inequality 
$\tau_k(\alpha_\ast) - e(\alpha_\ast)>0$ note that in the case $\tau_{k+1}(\alpha_\ast)\le 0$ we have 
$\tau_{k+1}(\alpha_\ast) < \tau_{k-1}(\alpha_\ast)$, and hence $\tau_k(\alpha_\ast) -e(\alpha_\ast)= \tau_{k-1}(\alpha_\ast)>0$.

We see that 
\begin{eqnarray*}f_{k+1}(Y) +f_{k-1}(Y) &\le& (1+t')\frac{n^{\tau_k(\alpha)-e(\alpha)}}{k!} +(1+t'')\frac{n^{\tau_k(\alpha)-e(\alpha)}}{(k+2)!}\\
&\le& n^{\tau_k(\alpha)-e(\alpha)}\frac{r+2}{(k+1)!},
\end{eqnarray*}
a.a.s. Finally, using (\ref{ineqs1}) we obtain
\begin{eqnarray}
b_k(Y) &\ge& \frac{n^{\tau_k(\alpha)}}{(k+1)!} \cdot \left[ 
1-t -(r+2)n^{-e(\alpha)}\right]\nonumber\\
&\ge& \frac{n^{\tau_k(\alpha)}}{(k+1)!} (1-o(1)),
\end{eqnarray}
a.a.s.
This proves inequality (\ref{firstly}). 

Finally to prove (\ref{dombeta1}) we observe that for $j\not=k$ we have (assuming that $\tau_j(\alpha_\ast)>0$) 
\begin{eqnarray*}
\frac{b_j(Y)}{b_k(Y)} &\le& \frac{f_j(Y)}{b_k(Y)} \\ 
&\le& \frac{(1+t)}{1-t'} \cdot \frac{(k+1)!}{(j+1)!} \cdot n^{\tau_j(\alpha)-\tau_k(\alpha)}\\
&\le&  (1+o(1))\cdot (r+1)!\cdot n^{-e(\alpha)},
\end{eqnarray*}
a.a.s., which is equivalent to (\ref{dombeta1}). Here $t, t'$ are sequences tending to $0$. In the case when $\tau_j(\alpha_\ast)\le 0$ we apply statement (b) of Theorem \ref{thmdim} taking $\mu$ such that $0\le \mu \le \tau_k(\alpha_\ast)-e(\alpha_\ast)$. Then 
\begin{eqnarray*}
\frac{b_j(Y)}{b_k(Y)} &\le& \frac{f_j(Y)}{b_k(Y)} \\ 
&\le & \frac{(k+1)!}{(1-t)}\cdot n^{\mu-\tau_k(\alpha)} \\
&\le&  (1+o(1))\cdot (r+1)!\cdot n^{-e(\alpha)}, 
\end{eqnarray*}
a.a.s., as above. This completes the proof.
\qed

\vskip 2cm

\section{Degrees of faces above and below the critical dimension in random complexes}


It this section we describe degrees of faces in random simplicial complexes. We show that their behaviour is very different depending on whether we are above or below the critical dimension. In dimensions $d$ below the critical dimension most $d$-dimensional faces have degree very close to $n^{1-\psi}$ with 
the exponent $\psi\in(0,1)$ depending on $d$ and $n$, while in dimensions above or at the critical dimension most faces have degree 0, i.e. are isolated.

Recall that {\it the degree} of a $d$-dimensional simplex $\sigma$ in a simplicial complex $Y$ is defined as the number of $(d+1)$-dimensional simplices in $Y$ that contain $\sigma$. We denote this number by $\deg_Y(\sigma)$. 

Given integers $d\ge 0$ and $s\ge 0$, we denote by $f_{d, s}(Y)$ the number of $d$-dimensional simplexes of $Y$ having degree $s$; note that 
$f_d(Y) = \sum_{s\ge 0} f_{d, s}(Y).$

If $Y\in \Omega_n^r$ is a random simplicial complex then 
$$f_{d, s}: \Omega_n^r\to \R$$
is a random variable which we shall  study in this section. 

\begin{lemma}\label{lm8} For any $d\le r$ and any $s\ge 0$ the expectation of $f_{d, s}$ is given by the formula
\begin{eqnarray}\label{efds}
\E(f_{d, s}) = \Lambda_{d, s}(\alpha) \cdot \E(f_d),
\end{eqnarray}
where 

\begin{eqnarray}\label{lambda}
\Lambda_{d, s}(\alpha) &=& \binom {n-d-1}{s} \cdot \lambda^s \cdot (1-\lambda)^{n-d-s-1},\\ \nonumber \\
 \lambda&=&n^{-\psi_{d+1}(\alpha)}
\end{eqnarray}
and $\E(f_d)$ is given by formula (\ref{efd}). 
\begin{proof}
Clearly we have
\begin{align*}
\E(f_{d,s}(Y))&=\sum_{\sigma\subset \Delta_n}  \PP_{r, \p}(\{ Y; \, \sigma\subset Y \quad \mbox{and}\quad  deg_Y(\sigma)=s\})\\
&=\binom{n}{d+1}\cdot \PP_r(\{Y; \sigma_0\subset Y \quad \mbox{and}\quad deg_Y(\sigma_0)=s\}),
\end{align*}
where in the first line $\sigma$ runs over all $d$-dimensional simplices of $\Delta_n$ and in the second line $\sigma_0$ denotes a fixed $d$-dimensional simplex. Here we use the obvious symmetry of the model. 
We restate the condition $deg_Y(\sigma_0)=s$ as $$f_0(\lk_Y(\sigma_0))=s$$ where $\lk_Y(\sigma_0)$ denotes the link of the $d$-simplex $\sigma_0$ in $Y$.

Let $\Lambda_{d,s}(\alpha)$ denote the conditional probability that a $d$-simplex $\sigma_0$ has degree $s$ in $Y$ given that $\sigma_0$ 
is contained in a random complex $Y\in\Omega_n^r$, i.e. 
$$\Lambda_{d,s}(\alpha)\, =\, \PP_{r, \p}\left(f_0(\lk_Y(\sigma_0))=s\ \, |\, \ \sigma_0\subset Y\right).$$
By the definition of conditional probability we have
$$\PP_{r, \p}\left(\{Y; \, \sigma_0\subset Y \quad \mbox{and}\quad \deg_Y(\sigma_0)=s\right\})\, =\, \PP_{r, \p}\left(\{Y; \sigma_0\subset Y\right\})\cdot \Lambda_{d,s}(\alpha).$$
We see that
$\E(f_{d,s}(Y))=\Lambda_{d,s}(\alpha)\cdot \E(f_d(Y)), $ i.e. (\ref{efds}) holds. By Lemma 3.2 of \cite{CF15} 
the link $\lk_Y(\sigma_0)$ is a multi-parameter random simplicial complex on $n-d-1$ vertices with the multi-parameter 
$(p_0', p_1', \dots, p'_{r-d-1})$ where 
$$p'_i= \prod_{j=i}^{i+d+1} p_j^{\binom {d+1}{j-i}}.$$ In particular, 
$$p_0' =\prod_{j=0}^{d+1}p_j^{\binom{d+1}{j}}=n^{-\sum_{j=0}^{d+1}\binom{d+1}{j}\alpha_j}=n^{-\psi_{d+1}(\alpha)}=\lambda.$$
This completes the proof. 
%
\end{proof}

\end{lemma}

\begin{corollary}\label{degree}
Let $Y\in\Omega_n^r$ be a random simplicial complex with respect to the probability multi-parameter $\p=(p_0,p_1,\ldots,p_r)$ where 
$p_i=n^{-\alpha_i}$ for $i=0,\ldots,r$ and the vector of exponents
$\alpha=\alpha(n)=(\alpha_0,\alpha_1,\ldots,\alpha_r)$
converges to a vector $\alpha_*\in\R^{r+1}_+$. Assume that $\alpha_\ast\in \D_k$ for some $k\geq 0$.
For an integer $0\leq d\leq r$ and a constant $0<\delta<1$ denote 
$$\mu=\mu(n, d, \alpha) = n^{1-\psi_{d+1}(\alpha)}$$
and
$$\mathcal{T}(n,\delta)=\{s\in \mathbb{N}\cup \{0\}; |s-\mu | > \delta\mu\}.$$
Then:

(1) Assume that $d<k$ and $\psi_{d+1}(\alpha_\ast)>0$. Then one has 
\begin{eqnarray}
\sum_{s\in \mathcal{T}(n,\delta)}f_{d,s}(Y)=0,
\end{eqnarray}
a.a.s. In other words, with probability tending to one, the degrees $s$ of all $d$-dimensional simplexes of a random complex $Y\in \Omega_n^r$ satisfy 
the inequality $|s-\mu|\le \delta \mu.$

(2) Assume now that $d\geq k$ and $\tau_d(\alpha_\ast)>0$. Then 
there exists a sequence 
 $t=t(n)> 0$ tending to $0$ such that 
\begin{eqnarray}\frac{f_{d,0}(Y)}{f_d(Y)}>1-t,\end{eqnarray}
a.a.s.
In other words, if $d$ is greater or equal than the critical dimension $k$,  \lq\lq almost all\rq\rq\,   simplexes of dimension $d$ have degree $0$ for $n\to \infty$. 
\end{corollary}
%
%
%
%
%
%

\begin{proof} Note that in the case (1) we have $\psi_{d+1}(\alpha_\ast)<1$ implying $\mu\to \infty$ and in the case (2) 
we have $\psi_{d+1}(\alpha_\ast)>1$ implying $\mu\to 0$. 

Since $f_{d,s}(Y)$ is a non-negative integer valued random variable, the statement (1) of the lemma is equivalent to
$$\PP_{r,\p}\left( \sum_{s\in \mathcal{T}(n,\delta)}f_{d,s}(Y)\geq 1\right)\to 0,\quad \mbox{when} \quad n\to\infty.$$
For technical reasons we introduce the set 
$$\mathcal{T}'(n,\delta)=\{s\in \mathbb{N}\cup \{0\}; |s-\mu' | > \frac{1}{2} \delta\mu'\},$$
where 
$$\mu'=(n-d-1)\cdot n^{-\psi_{d+1}(\alpha)} =\mu- (d+1)n^{-\psi_{d+1}(\alpha)}.$$
One easily checks that $\T(n, \delta) \subset \T'(n, \delta)$. 
By the Markov inequality and Lemma \ref{lm8} we have
\begin{eqnarray*}
\PP_{r,\p}\left(\sum_{s\in \mathcal{T}(n,\delta)}f_{d,s}(Y)\geq 1\right)&\le& \PP_{r,\p}\left(\sum_{s\in \mathcal{T}'(n,\delta)}f_{d,s}(Y)\geq 1\right)\\
&\leq& \E\left( \sum_{s\in \mathcal{T}'(n,\delta)}f_{d,s}\right)\\
&=& \sum_{s\in \mathcal{T}'(n,\delta)}\E(f_{d,s})\\
&=&\E(f_d)\cdot \sum_{s\in \mathcal{T}'(n,\delta)}\Lambda_{d,s}(\alpha)\\
&\leq& n^{d+1}\cdot\left(\sum_{s\in \mathcal{T}'(n,\delta)}\Lambda_{d,s}(\alpha)\right).
\end{eqnarray*}
Recall that $\Lambda_{d,s}(\alpha)$ is binomially distributed with $n-d-1$ trials and with the probability parameter $\lambda=n^{-\psi_{d+1}(\alpha)}$, see 
(\ref{lambda}).
By the well-known Chernoff bound (see \cite{JLR}, formula (2.9) on page 27) we have
$$\sum_{s\in \mathcal{T}'(n,\delta)}\Lambda_{d,s}(\alpha)\leq 2e^{-\delta^2\mu/3},$$
which implies that
$$\PP_{r,\p}\left( \sum_{s\in \mathcal{T}(n,\delta)}f_{d,s}(Y)\geq 1\right)\leq 2n^{d+1}e^{-\delta^2\mu/3}.$$
We claim that $n^{d+1}e^{-\delta^2\mu/3}$ tends to $0$ as $n\to \infty$. 
Indeed, for $n$ large enough, 
 $$\log (n^{d+1}e^{-\delta^2\mu/3})<(d+1)\log n- n^{(1-\psi_{d+1}(\alpha_*))/2}\to -\infty$$
 since $1-\psi_{d+1}(\alpha_*)>0$. This proves statement (1).

Next we prove the statement (2) of the Lemma. Let $\omega=\omega(n)$ be a sequence tending to $\infty$ and let $t=\omega\mu$. 
Since $d\geq k$, where $k$ is the critical dimension, we have $$\mu = n^{1-\psi_{d+1}(\alpha)}\to 0.$$
Define the subset $\Omega'\subset \Omega_n^r$ as follows
$$\Omega'=\left\{ Y\in\Omega_n^r\ |\ f_d(Y)\geq\frac{\E (f_d)}{2} \right\}.$$
By Theorem \ref{thmdim}, one has $\PP_{r,\p}(Y\in\Omega')\to 1$ as $n\to\infty$ 
under the condition $\tau_{d}(\alpha_*)>0$. 
Hence, by the Markov inequality and Lemma \ref{lm8} we obtain
\begin{eqnarray*}
\PP_{r,\p}\left(\frac{f_{d,0}(Y)}{f_d(Y)}<1-t\right)
&\leq&\PP_{r,\p}\left(\frac{f_{d,0}(Y)}{f_d(Y)}<1-t\quad \mbox{and}\quad Y\in\Omega'\right)+\PP_{r,\p}\left( Y\not\in\Omega'\right)\\
&=&\PP_{r,\p}\left(\sum_{s\geq 1}\frac{f_{d,s}(Y)}{f_d(Y)}\geq t\quad \mbox{and}\quad Y\in\Omega'\right)+\PP_{r,\p}\left( Y\not\in\Omega'\right)\\
&\leq&\PP_{r,\p}\left(\sum_{s\geq 1}f_{d,s}(Y)\geq t\cdot\frac{\E(f_d)}{2}\right)+o(1)\\
&\leq& \frac{2\E\left(\sum_{s\geq 1}f_{d,s}\right)}{t\E(f_d)}+o(1)\\
&=& \frac{2}{t}\cdot\sum_{s\geq 1}\Lambda_{d,s}+o(1)
\end{eqnarray*}
It is obvious from the definition of $\Lambda_{d,s}$ that $\Lambda_{d,s}\leq \mu^s$ and since $\mu\to 0$ we see that for $n$ large enough $$\sum_{s\geq 1}\Lambda_{d,s}=\sum_{s\geq 1}\mu^s \, \leq \, 2 \mu=2 \frac{t}{\omega}.$$
In particular
$$\PP_{r,\p}\left(\frac{f_{d,0}(Y)}{f_d(Y)}<1-t\right)\leq\frac{4}{\omega}+o(1)\to 0.$$

\end{proof}

A simplicial complex $X$ is said to be {\it pure} if $f_{d,0}(X)=0$ for every $d<\dim X$.
The following result is an obvious corollary of the first part of Lemma \ref{degree}. It was also proven in \cite{Fowler}, Lemma 9. 
\begin{corollary}\label{corpure}
Let $Y\in\Omega_n^r$ be a random simplicial complex with respect to the probability multi-parameter $\p=(p_0,p_1,\ldots,p_r)$ where 
$p_i=n^{-\alpha_i}$ for $i=0,\ldots,r$ and the vector of exponents
$\alpha=\alpha(n)=(\alpha_0,\alpha_1,\ldots,\alpha_r)$
converges to a vector $\alpha_*\in\R^{r+1}_+$. Assume that $\alpha_\ast\in \D_k$ for some $k\geq 0$. 
Then with probability tending to one the $k$-dimensional skeleton $Y^{(k)}$ of $Y$ is pure, a.a.s..
\end{corollary}
\begin{proof}
By statement (1) of Corollary \ref{degree} we have $f_{d,0}(Y)=0$, a.a.s. for every integer $0\leq d<k$. 
Thus, $$\sum_{d=0}^{k-1}f_{d,0}(Y)=0,\quad \mbox{a.a.s.},$$
i.e. with probability tending to one $Y$ has no degree zero simplices of dimension less than $k$.
\end{proof}

\section{Betti numbers and the fundamental group of random simplicial complexes}

The main result of this section states that random simplicial complexes have trivial rational homology below the critical dimension. This result was proven in \cite{Fowler}, Theorem 2. We also show that the fundamental groups of random simplicial complexes  have property (T) if the critical dimension equals $2$ and are trivial if the critical dimension is greater than $2$. 

\begin{theorem}\label{thmgarland}
Let $Y\in \Omega_n^r$ be a random simplicial complex with respect to the multi-parameter 
$\mathfrak p=(p_0, p_1, \dots, p_r)=n^{-\alpha}$ where $\alpha=\alpha(n)\in \R_+^{r+1}$ 
is a sequence of vectors converging to a vector $\alpha_*\in\mathfrak{D}_k$, where $2\le k\le r$. Then for every $0<j\leq k-1$ one has
$$H_{j}(Y;\mathbb Q)=0,$$
a.a.s.
\end{theorem}


Given a graph $G$ we denote by $\mathcal{L}=\mathcal{L}(G)$ the normalized Laplacian of $G$, see \cite{Chung}. All the eigenvalues of 
$\mathcal{L}$ lie in the interval $[0,2]$ and the multiplicity of $0$ as an eigenvalue is equal to the number of connected components of $G$.
A quantity of special interest is the smallest non-zero eigenvalue $\kappa(\mathcal{L})>0$ of $\mathcal{L}$, also known as {\it the spectral gap} of $G$. 

Let $X$ be finite simplicial complex. 
Each $\ell$-dimensional simplex $\sigma\subset X$ determines a graph $L_\sigma$ which is defined as the 1-skeleton of the link $\lk_X(\sigma)$. 
The vertices of $L_\sigma$ are simplexes of dimension $\ell+1$ containing $\sigma$ and the edges of $L_\sigma$ are the simplexes of dimension $\ell+2$
containing $\sigma$. 

\begin{lemma}\label{implies} Let $Y\in \Omega_n^r$ be a random simplicial complex with respect to the multi-parameter 
$\mathfrak p=n^{-\alpha}$ where $\alpha=\alpha(n)$ 
is a sequence of vectors in $\R_+^{r+1}$ converging to a vector $\alpha_*\in\mathfrak{D}_k$, where $2\le k\le r$. Then, with probability tending to one, 
$Y$ has the following property: the graph $L_\sigma$ associated to every $\ell$-dimensional simplex $\sigma\subset Y$, 
where $0\leq\ell\leq k-2$, is a non-empty connected graph with the spectral gap satisfying
 $$\kappa(\mathcal{L}_\sigma)>1-\frac{1}{\ell+2}.$$
\end{lemma}

Theorem \ref{thmgarland} obviously follows from Lemma \ref{implies} by applying the well-known result \cite{BS} to the skeleta $Y^{(\ell+2)}$ of a random complex
$Y$, where $\ell\le k-1$. 

\begin{theorem}[Ballman--\'{S}wi\c{a}tkowski \cite{BS}]
Let $\ell\geq 0$ be a non-negative integer. If $X$ is a finite $(\ell+2)$-dimensional simplicial complex such that for every $\ell$-dimensional simplex $\sigma\subset X$ the link $L_\sigma=\lk_X(\sigma)$ is non-empty and connected and 
the spectral gap of the normalized Laplacian $\mathcal{L}_\sigma$ of the link $L_\sigma$ satisfies $$\kappa(\mathcal L_\sigma)>1-\frac {1}{\ell+2},$$ then $$H^{\ell+1}(X;\mathbb Q)=0.$$
\end{theorem}

\begin{proof}[Proof of Lemma \ref{implies}]
%
%
%
%
%
%
%
%
%
Let $\sigma^{(\ell)}\subset \Delta_n$ be an $\ell$-simplex. We shall consider all random simplicial complexes $Y\in \Omega_n^r$ containing $\sigma$, where $\ell\le k-2$. 
By Lemma 4.1 of \cite{CF15}, the link $L_\sigma=\lk_Y(\sigma)$ of $\sigma$ in $Y^{(\ell+2)}$ is a random simplicial complex of dimension one (a random graph),
$L_\sigma\in\Omega_{n-\ell-1}^{1}$, with respect to the multi-parameter
$\mathfrak{p}'=(p_0',p_1')$
where
$$p_0'=\prod_{j=0}^{\ell+1}p_j^{\binom{\ell+1}{j}}=n^{-\psi_{\ell+1}(\alpha)}$$
and
$$p_1'=\prod_{j=1}^{\ell+2}p_j^{\binom{\ell+1}{j-1}}=\prod_{j=1}^{\ell+2}p_j^{\left[\binom{\ell+2}{j}-\binom{\ell+1}{j}\right]}=n^{\psi_{\ell+1}(\alpha)-\psi_{\ell+2}(\alpha)}.$$
In other words,
$$\frac{\PP_{r, \p}(Y\supset \sigma\quad \mbox{and}\quad \lk_Y(\sigma)=L)}{\PP_{r, \p}(Y\supset \sigma)} = \PP_{1, \p'}(L).$$


Denote $$\mu =(n-\ell-1)p_0';$$ clearly, 
$\mu$ is the expected number of vertices $f_0(L_\sigma)$ in the random graph $L_\sigma$, which coincides with the expected degree of a simplex of dimension $\ell$  in $Y$. 
Recall that $\alpha_\ast\in \D_k$ implies $\psi_k(\alpha_\ast) <1$ (see (\ref{defdk})) and hence for $n$ sufficiently large one has 
$\psi_{\ell+1}(\alpha_\ast)<\psi_{\ell+2}(\alpha_\ast)<1.$ 
Therefore we have
$$\mu\sim n^{1-\psi_{\ell+1}(\alpha)}\to\infty.$$
We may apply Lemma 2.9 of \cite{CF15}  
 to obtain that, conditioned to $\sigma^{(\ell)}$ being a simplex of the random complex $Y$, for any $\epsilon>0$ and any $0<\epsilon'<1/2$ there exists an integer $N_{\epsilon,\epsilon'}$ such that for all $n>N_{\epsilon,\epsilon'}$ 
 one has
\begin{eqnarray}\label{vertices1}
(1-\epsilon)\mu\leq f_0(L_\sigma)\leq (1+\epsilon)\mu
\end{eqnarray}
 with probability
$$\geq 1-2e^{-\frac{1}{3}\mu^{2\epsilon'}}.$$
A first moment argument shows that (\ref{vertices1}) holds \textit{simultaneously} for all $\ell$-simplices in $Y$, a.a.s.
Indeed, if $\eta_\ell$ is the random variable that counts the number of $\ell$-simplices of $Y$ for which inequality (\ref{vertices1}) fails then for some 
$\epsilon''>0$,
\begin{eqnarray*}
\E(\eta_\ell)&\leq& n^{\ell+1}\cdot2\exp\left(-\frac{\mu^{2\epsilon'}}{3}\right)\\
&\leq& Cn^{\ell+1}\exp(-n^{\epsilon''})\\
&=& o(1).
\end{eqnarray*}
We conclude that for any $\epsilon>0$ a random complex $Y\in\Omega_n^r$ with probability tending to one 
%
 has the property that for every $\sigma^{(\ell)}\subset Y$ the graph $L_\sigma$ satisfies
\begin{eqnarray*}
\frac{f_0(L_\sigma)p_1'}{\log(f_0(L_\sigma))}&>&\frac{(1-\epsilon)(n-\ell-1)p_0'p_1'}{\log((1+\epsilon)(n-\ell-1)p_0')}\\
&=&\frac{(1-\epsilon)(n-\ell-1)n^{-\psi_{\ell+2}(\alpha)}}{\log((1+\epsilon)(n-\ell-1)n^{-\psi_{\ell+1}(\alpha)})}\\
&>&C_{\epsilon,\ell}\frac{n^{1-\psi_{\ell+2}(\alpha)}}{(1-\psi_{\ell+1}(\alpha))\log n} \, \, \, \to \, \, \,  \infty,
\end{eqnarray*}
where $C_{\epsilon,\ell}>0$ is a constant dependent only on $\epsilon$ and $\ell$.
%
Hence for any real number $\delta>0$ and $n$ large enough (depending on $\delta$) we obtain
\begin{eqnarray}\label{p1'}
p_1'>\frac{(1+\delta)\log(f_0(L_\sigma))}{f_0(L_\sigma)}
\end{eqnarray}
for every $\ell$-simplex $\sigma^{(\ell)}\subset Y$, a.a.s.

Now we shall use a simplified version of Theorem 1.1 from \cite{HKP} (as appears in \cite{Fowler}) which states the following.  
Fix $\delta>0$ and let
$
p\geq \frac{(1+\delta)\log n}{n}.
$
Let $G\in G(n,p)$ be a Erd\H{o}s--R\'{e}nyi random graph. Then for any $c>1$ there exists $N_{c,\delta}>0$ (depending on $c$ and $\delta$) such that for $n>N_{c,\delta}$ the graph $G$ is connected and one has
$$\kappa(G)> 1-\frac{1}{c}$$
with probability at least $1-n^{-\delta}$.
We apply this result to Theorem the link $L_\sigma$ of every  $\ell$-simplex $\sigma\subset Y$.
%
Choosing $\delta=c=\ell+2$ and applying this theorem to the graphs $L_\sigma$ gives the bound
\begin{eqnarray}\label{spectralbound}
\PP_{r, \p}\left(b_0(L_\sigma)>1 \quad \mbox{or} \quad \min_{\sigma^{(\ell)}\subset Y}\kappa(\mathcal{L}_\sigma)\leq \frac{1}{\ell+2}\right)=o(n^{-(\ell+1)})
\end{eqnarray}
Hence the expected number of $\ell$-simplexes $\sigma$ in $Y$ with either disconnected $L_\sigma$ or such that 
$\kappa(\mathcal{L}_\sigma)\leq \frac{1}{\ell+2}$ tends to zero and by the Markov inequality it follows
that with probability tending to one $Y$ satisfies the property that all graphs $L_\sigma$ are connected and 
$$\kappa(\mathcal{L}_\sigma)>1-\frac{1}{\ell+2}$$ for every $\ell$-simplex $\sigma\subset Y$. 
This completes the proof of Lemma \ref{implies}. 
%
%
%
%
%
%
%
%
%
\end{proof}

Next we state a few related results. 

\begin{theorem}\label{k1}
Let $Y\in \Omega_n^r$ be a random simplicial complex with respect to the multi-parameter 
$\mathfrak p=(p_0, p_1, \dots, p_r)=n^{-\alpha}$ where $\alpha=\alpha(n)\in \R_+^{r+1}$ 
is a sequence of vectors converging to a vector $\alpha_*\in\mathfrak{D}_k$ where $1\le k\le r$. Then 
$b_0(Y)=1$, i.e. $Y$ is connected, 
a.a.s.
\end{theorem}
\begin{proof} We want to show that the conditions of Corollary 7.2 from \cite{CF15} are satisfied, i.e. 
$np_0p_1-\log (np_0) \to \infty.$
This is equivalent to 
\begin{eqnarray}\label{obvious}
n^{1-\alpha_0-\alpha_1}-(1-\alpha_0)\log n \to \infty.\end{eqnarray}
Since the vector of exponents $\alpha$ converges to $\alpha_\ast\in \D_k$, where $k\ge 1$, we have 
$\psi_1(\alpha_\ast)<1$ and therefore $1-\alpha_0-\alpha_1>(1-\alpha_{\ast 0}-\alpha_{\ast 1})/2 =\delta>0$ for all large enough $n$.
This shows that (\ref{obvious}) is obviously satisfied. 
\end{proof}

\begin{theorem}\label{k2}
Let $Y\in \Omega_n^r$ be a random simplicial complex with respect to the multi-parameter 
$\mathfrak p=(p_0, p_1, \dots, p_r)=n^{-\alpha}$ where $\alpha=\alpha(n)\in \R_+^{r+1}$ 
is a sequence of vectors converging to a vector $\alpha_*\in\mathfrak{D}_k$ with $k\ge 2$. Then the fundamental group
$\pi_1(Y)$ 
has property (T),
a.a.s.
\end{theorem}
To show that $\pi_1(Y)$ has property (T) (a.a.s.) we apply Lemma \ref{implies} and 
the following result of A. { \. Z}uk:

\begin{theorem}[\.{Z}uk, \cite{Z}]
If $X$ is a finite, pure $2$-dimensional simplicial complex so that for every vertex $v$ the link $L_v$ is connected and the normalised Laplacian 
$\mathcal{L}=\mathcal{L}(L_v)$ satisfies $\kappa(\mathcal{L})>1/2$. Then $\pi_1(X)$ has property (T).
\end{theorem}

Note that the domain 
$\D_2$ is given by the inequalities
$$\alpha_0+2\alpha_1+\alpha_2 <1< \alpha_0+3\alpha_1+3\alpha_2+\alpha_3;$$
it can be 
decomposed into two subdomains 
$$\alpha_0+3\alpha_1+2\alpha_2 <1< \alpha_0+3\alpha_1+3\alpha_2+\alpha_3,$$
where the  fundamental group $\pi_1(Y)$ is trivial, and 
$$\alpha_0+2\alpha_1+\alpha_2 <1< \alpha_0+3\alpha_1+2\alpha_2,$$
where the  fundamental group $\pi_1(Y)$ is nontrivial, see Theorem 8.1 from \cite{CF15}. See Figure \ref{domaind2}. 
Note that in the second subdomain the first homology group with rational coefficients $H_1(\pi_1(Y); \Q)=0$ is trivial, i.e. we are dealing with a nontrivial random perfect group. 
 \begin{figure}[h]
\centering
\includegraphics[width=0.5\textwidth]{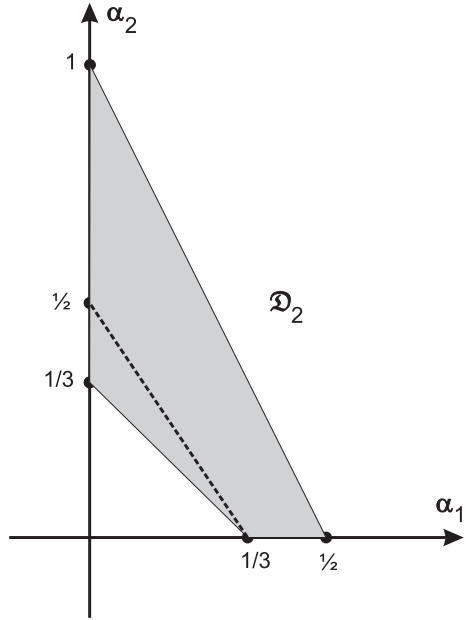}
\caption{Intersections of the domain $\D_2$ with the plane $\alpha_0=\alpha_3=\dots=0$ and its two subdomains. The simple connectivity subdomain lies below the dotted line. }\label{domaind2}
\end{figure}

Finally we show that random complexes are simply-connected assuming that $\alpha_\ast\in \D_k$ for $k\ge 3$:

\begin{theorem}\label{kgreater3}
Let $Y\in \Omega_n^r$ be a random simplicial complex with respect to the multi-parameter 
$\mathfrak p=(p_0, p_1, \dots, p_r)=n^{-\alpha}$, where $\alpha=\alpha(n)\in \R_+^{r+1}$ 
is a sequence of vectors converging to a vector $\alpha_*\in\mathfrak{D}_k$, where $3\le k\le r$. Then 
$\pi_1(Y)=1,$ i.e. $Y$ is simply connected, 
a.a.s.
\end{theorem}

\begin{proof} We apply Theorem 8.1 from \cite{CF15}. Since $\alpha_\ast\in \D_k$ where $k\ge 3$ we have 
$\psi_3(\alpha_\ast)<1$, i.e. $\alpha_{\ast 0}+3\alpha_{\ast 1}+3\alpha_{\ast 2} + \alpha_{\ast 3}<1$ implying that 
$$\alpha_{\ast 0}+3\alpha_{\ast 1}+2\alpha_{\ast  2}<1.$$ 
Denoting $1- \alpha_{\ast 0}- 3\alpha_{\ast 1}- 2\alpha_{\ast  2}= 2\delta$ we see that 
$$1- \alpha_{ 0}+3\alpha_{1}- 2\alpha_{ 2}>\delta>0$$
for all sufficiently large $n$. Now one checks that the relations (35), (36), (37) of Theorem 8.1 from \cite{CF15} are obviously satisfied. 
This completes the proof. 
\end{proof}

{\bf Question:} { \it Is it true that for $\alpha_\ast\in \D_1$ a random complex $Y\in \Omega_n^r$ has a free fundamental group, a.a.s.? }

We know that it is true in two important special cases: in the Linial-Meshulam model (see \cite{CCFK}, Theorem 1) and in the case of clique complexes of random graphs
(see \cite{CFH}, Theorem A). In the Linial - Meshulam case one has $r=2$ and the probability multi-parameter has the exponent 
$(0,0, \alpha_2)$; the domain $\D_1$ is given by the inequality $\alpha_2>1$. In the case of clique complexes of random graphs, the exponent of the probability multi-parameter is $(0, \alpha_1, 0, 0, \dots)$ and the domain $\D_1$ is described by the inequality $1/2< \alpha_1<1$. 
\vskip 0.4 cm

{\bf Question:} Another interesting open question is {\it whether, on certain subdomain of $\D_k$, any finite simplicial complex $S$ of dimension $\le k$ admits a topological embedding into a random simplicial complex $Y\in \Omega_n^r$, a.a.s. }
For the Linial-Meshulam model the corresponding subdomain was described in Theorem 3 from \cite{CCFK}. By a topological embedding we understand a simplicial embedding of a subdivision of $S$ into $Y$. 

\section{Minimal cycles and spheres above the critical dimension}

By a $k$-dimensional minimal cycle $S$ we understand a $k$-dimensional simplicial complex $S$ such that $b_k(S)=1$ and $b_k(S')=0$ for any proper subcomplex. 
In this section we examine minimal cycles contained in multi-parameter random simplicial complexes. 
We 
show that in dimensions above the critical dimension there may exist only "small" minimal cycles, i.e. having limited number of vertices. 
Next 
we apply the Dehn -- Sommerville relations to give an improved estimate of "the size" of spheres 
contained in random simplicial complexes.

Let $Y$ be a random complex with respect to the multi-parameter $\p=n^{-\alpha}$ where $\alpha=(\alpha_0, \alpha_1, \dots, \alpha_r)\in \R^{r+1}_+$. 
In this section we assume (for the sake of simplicity) that the multi-exponent $\alpha\in  \R^{r+1}_+$ is constant, i.e. it does not depend on $n$. 

Let $k$ be an integer lying above the critical dimension and not exceeding the real dimension;
this is equivalent to the assumptions:
\begin{eqnarray}\label{between}
\psi_k(\alpha) >1 \quad\mbox{and}\quad \phi_k(\alpha)<1.\end{eqnarray}
Recall the notations we use
$$\psi_k(\alpha)=\sum_{i=0}^r \binom k i \alpha_i \quad \mbox{and}\quad \phi_k(\alpha)=\sum_{i=0}^r \binom k i \frac{\alpha_i}{i+1}.$$
We shall use the relation
\begin{eqnarray}\label{rel}
(k+1)\cdot \phi_k(\alpha) \, =\,  \sum_{i=0}^k \psi_i(\alpha),
\end{eqnarray}
which follows from (\ref{sumpsi1}). The following result is also contained in \cite{Fowler}. 

\begin{theorem}\label{smallcycles} Under the above assumptions, 
a random simplicial complex $Y\in \Omega_n^r$ contains no strongly connected\footnote{Recall that a simplicial complex $S$ of dimension $k$ is strongly connected if the space $S-S^{(k-2)}$ is path-connected.} $k$-dimensional subcomplexes $S$ with 
\begin{eqnarray}\label{f0} f_0(S) > (k+1)\left[1+\frac{1-\phi_k(\alpha)}{\psi_k(\alpha)-1}\right]\end{eqnarray}
vertices, a.a.s. In particular, with probability tending to one as $n\to \infty$, a random complex $Y\in \Omega_n^r$ has the following property:
any $k$-dimensional minimal cycle $S$ contained in $Y$ must satisfy
\begin{eqnarray}\label{f0yclec} f_0(S) \le (k+1)\left[1+\frac{1-\phi_k(\alpha)}{\psi_k(\alpha)-1}\right].\end{eqnarray}
%
%
\end{theorem}
\begin{proof}
Let $S$ be a strongly connected $k$-dimensional simplicial complex. Then 
$f_0(S)= k+1 +x,$ where $x\ge 0$, and one shows by induction on the number of $k$-simplexes in $S$ that 
$$f_i(S) \ge \binom {k+1} {i+1} +x \cdot \binom k i, \quad \mbox{where}\quad i=0, 1, \dots, k.$$
We know that the probability that $Y$ contains a subcomplex isomorphic to $S$ is less than or equal to 
\begin{eqnarray*}
n^{f_0(S)} \cdot \prod_{i=0}^r p_i^{f_i(S)} \le \\
n^{k+1+x}\cdot \prod_{i=0}^r p_i^{\binom {k+1} {i+1} + x\cdot \binom k i }= \\
n^{(k+1)\left[1-\phi_k(\alpha)\right]-x(\psi_k(\alpha)-1)}
\end{eqnarray*}
where we have used the identity (\ref{sumpsi1}). 
The exponent is negative if 
$$x> (k+1)\cdot \frac{1-\phi_k(\alpha)}{\psi_k(\alpha)-1},$$
which is equivalent to (\ref{f0}). 

For a fixed $N$ there are only finitely many isomorphism types of simplicial complexes $S$ on $N=f_0(S)$ vertices and the estimate above applies to each of them. 
Hence, for a fixed $$N>(k+1)\cdot \left[1+ \frac{1-\phi_k(\alpha)}{\psi_k(\alpha)-1}\right]$$ a random simplicial complex $Y\in \Omega_n^r$ 
contains no strongly connected 
subcomplexes
$S$ with $f_0(S)=N$. 
Finally, any strongly connected $k$-dimensional simplicial complex on more than $N$ vertices contains as a subcomplex a strongly connected $k$-dimensional simplicial complex on $N$ vertices, and Theorem \ref{smallcycles} follows. 
\end{proof}

Next we give a stronger statement for containment of spheres of dimensions above the critical dimension. 

\begin{theorem}\label{smallspheres} 
Let $Y$ be a random complex with respect to the multi-parameter $\p=n^{-\alpha}$ 
where 
 the vector of multi-exponent $\alpha=(\alpha_0, \alpha_1, \dots, \alpha_r)\in  \R^{r+1}_+$ is constant, i.e. does not depend on $n$. 
Suppose that an integer $k\ge 2$ is above the critical dimension and does not exceed the real dimension, i.e. inequalities (\ref{between}) are satisfied. 
Let $S$ be a $k$-dimensional simplicial sphere satisfying 
\begin{eqnarray}\label{f00}
f_0(S) > (k+1)\cdot \left[1+ \frac{1-\phi_k(\alpha)-\alpha_k\cdot (k+1)^{-1}}{\psi_k(\alpha)-1+(\alpha_{k-1}+\alpha_k)}\right].
\end{eqnarray}
Then $S$ cannot be simplicially embedded into $Y$, a.a.s.
\end{theorem}

Note that (\ref{f00}) gives potentially stronger result compared with (\ref{f0}); indeed, we subtract a non-negative quantity 
$\alpha_k\cdot (k+1)^{-1}$
in the numerator and add a non-negative quantity $\alpha_{k-1}+\alpha_k$ in the denominator. 

\begin{proof}
Let $S$ be a pure, connected, $k$-dimensional simplicial complex.
Recall the notions of the {\it $f$-polynomial} and the $h$-polynomials associated to $S$. 
The $f$-polynomial is defined as
\begin{eqnarray}
f(t)=f_{-1}t^{k+1}+f_{0}t^{k}+f_{1}t^{k-1}+\ldots +f_{k-1}t+f_{k}.
\end{eqnarray}
where $f_{-1}=1$ and $f_i=f_i(S)$ are the face numbers of $S$. 
The {\it $h$-polynomial} associated to $S$ is defined by $$h(t)=f(t-1),$$ i.e.
\begin{eqnarray}
h(t)=h_{0}t^{k+1}+h_{1}t^{k}+\ldots +h_{k}t+h_{k+1}
\end{eqnarray}
where
\begin{eqnarray}\label{ftoh}
h_i=\sum_{j=0}^i (-1)^{i-j}\binom{k-j+1}{k-i+1}f_{j-1},\quad 0\leq i\leq k+1.
\end{eqnarray}
For example, one has $$h_0=1\quad \mbox{and}\quad h_1=f_0-k-1.$$
The system of linear equations (\ref{ftoh}) can be inverted to obtain $f_i$ as a nonnegative linear combinations of $h_0,\ldots, h_{i+1}$:
\begin{eqnarray}\label{htof}
f_{i-1}=\sum_{j=0}^i \binom{k-j+1}{k-i+1}h_j, \quad 0\leq i\leq k+1.
\end{eqnarray}

Let $Y\in\Omega_n^r$ be a random complex with respect to the probability multi-parameter $\mathfrak{p}=n^{-\alpha}$, where $\alpha=(\alpha_0,\alpha_1,\ldots,\alpha_r)$. In this section we assume that the vector of multi-exponents $\alpha$ is constant, i.e. it is independent of $n$.

Consider a fixed simplicial complex $S$ of dimension $k\leq r$ and let $f_i=f_i(S)$ denote the number of $i$-simplices in $X$.
By Theorem 1 from \cite{CF15a}, the condition
\begin{eqnarray}\label{nonemb}
\sum_{i=0}^{k}\alpha_if_i(S) \, >\, f_0(S)
\end{eqnarray}
implies then $S$ admits no embedding into $Y$, a.a.s. We want to express this condition in terms of the coefficients of the $h$-polynomial of $S$. 
We may use (\ref{htof}) and the equality $h_0=1$ to rewrite the sum $\sum_{i=0}^{d-1}\alpha_if_i(S)$ as
\begin{eqnarray*}
\sum_{i=0}^{k}\alpha_if_i(S) &=&\sum_{i=0}^{k}\alpha_i\left[\sum_{j=0}^{i+1}\binom{k-j+1}{i-j+1}h_j(S)\right]\\
&=&\sum_{j=0}^{k+1}h_j(S)\cdot\left[\sum_{i=j-1}^{k}\binom{k-j+1}{i-j+1}\alpha_i\right]\\
&=&\sum_{j=0}^{k+1}h_j(S)\cdot\gamma_j
\end{eqnarray*}
where we denote 
$$\gamma_j = \gamma_j(k,\alpha)=\sum_{i=j-1}^{k}\binom{k-j+1}{i-j+1}\alpha_i \, \ge 0,$$
where $j=0,1, \dots, k+1.$
Here we assume that $\alpha_{-1}=0$ by convention.
We obtain the following corollary:
{\it  Let $S$ be a $k$-dimensional simplicial complex where $k\le r$. If
\begin{eqnarray}\label{nonembed}
\sum_{j=0}^{k+1}  h_j(S)\cdot \gamma_j \, >\, h_1(S) +k+1
\end{eqnarray}
then $S$ is not embeddable into a random simplicial complex $Y\in \Omega_n^r$, a.a.s., 
with respect to the multi-parameter measure $\PP_{r, \p}$ where
$\p=n^{-\alpha}$, $\alpha=(\alpha_0, \dots, \alpha_r)$. 
}

We compute a few quantities $\gamma_j$:
\begin{eqnarray*}
\gamma_0&=&\sum_{i=0}^{k}\binom{k+1}{i+1}\alpha_i= (k+1)\cdot \phi_k(\alpha), \\ 
\gamma_1&=&\sum_{i=0}^{k}\binom{k}{i}\alpha_i=\psi_{k}(\alpha),\\
\gamma_{k}&=&\alpha_{k-1}+\alpha_{k},\\
\gamma_{k+1}&=&\alpha_{k}.
\end{eqnarray*}
%
%
%

If $S$ is a triangulated $k$-dimensional sphere then the Dehn -- Sommerville equations hold
\begin{eqnarray}\label{ds}
h_i(S) \, =\,  h_{k+1-i}(S),\quad \mbox{for} \quad 0\leq i\leq k+1.\end{eqnarray}
It is also well known (see \cite{St}) that
$h_i(S) \, \geq 0.$
Using the positivity $h_i(S)\ge 0$,  $\gamma_i\ge 0$ and the symmetry (\ref{ds}), we may simplify the above non-embedability criterion 
(\ref{nonembed}) by retaining only the terms containing $h_0, h_1, h_k, h_{k+1}$. We obtain the following statement:

{\it  Let $S$ be a $k$-dimensional simplicial sphere, where $k\le r$. If
\begin{eqnarray}\label{f01}
\gamma_0 +\gamma_{k+1} + (\gamma_1+ \gamma_k)\cdot h_1(S)  \, >\, h_1(S) +k+1,
\end{eqnarray}
then $S$ is not embeddable into a random simplicial complex $Y\in \Omega_n^r$, a.a.s.
}
We may rewrite (\ref{f01}) as 
$$h_1(S)> \frac{k+1-\gamma_0-\gamma_{k+1}}{\gamma_1+\gamma_k-1}$$
observing that $\gamma_1+\gamma_k-1>0$ since $\gamma_1=\psi_k(\alpha)>1$ according to (\ref{between}). Substituting here 
$h_1(S) =f_0(S) -k -1$ and expressing the values of $\gamma_0, \gamma_1, \gamma_k, \gamma_{k+1}$ in terms of $\alpha_i$ 
(see above) we obtain (\ref{f00}). 
\end{proof}

\bibliographystyle{amsalpha}

\end{document}